\documentclass[10pt,reqno]{amsart}
\usepackage{tikz}
\usepackage{comment}% http://ctan.org/pkg/comment
\usetikzlibrary{quotes}
\usetikzlibrary{arrows}

\tikzset{
    vertex/.style={circle,draw,minimum size=1.5em},
    edge/.style={->,> = latex'}
}
\usepackage{amsthm,amsmath,amssymb}
\usepackage[english]{babel}
\usepackage[utf8]{inputenc}
\usepackage{amsthm}
\usepackage{array}
\usepackage{framed}
\usepackage{pgfplots}
\usepackage{float}
\usepackage{graphicx}
\usepackage{amssymb}
\usepackage[inline]{enumitem}
\usepackage{thmtools}

\usepackage{array, makecell} % https://tex.stackexchange.com/a/280006/5549

\usepgfplotslibrary{fillbetween}
\pgfplotsset{compat=1.17}

\usepackage{mathtools}
\usepackage[colorlinks=true,backref=page,bookmarksopen=true]{hyperref} % Links in the pdf. Backref option: each bibliographical entry denotes where it was cited
\usepackage[capitalise]{cleveref} % Must be loaded after hyperref. It's advisable to compile the file with LUALATEX instead of PDFLATEX to avoid some annoying warning messages.
\usepackage[font=small]{caption}

\numberwithin{equation}{section}     % Makes labeled equations easier to find.

\declaretheorem[numberwithin=section]{theorem} \declaretheorem[sibling=theorem]{proposition} 
\declaretheorem[sibling=theorem]{corollary} 
\declaretheorem[sibling=theorem]{lemma}

\declaretheorem[sibling=theorem]{question}

\declaretheorem[sibling=theorem, style=definition]{definition}

\usepackage{xcolor}
\hypersetup{
    colorlinks,
    linkcolor={red!50!black},
    citecolor={blue!50!black},
    urlcolor={blue!80!black}
}

\renewcommand*{\backref}[1]{}
\renewcommand*{\backrefalt}[4]{\tiny 
  \ifcase #1 (\textbf{NOT CITED.})%
  \or    (Cited on page~#2.)%
  \else   (Cited on pages~#2.)%
  \fi}

\def\env@matrix{\hskip -\arraycolsep
  \let\@ifnextchar\new@ifnextchar
  \array{*\c@MaxMatrixCols c}}
  \makeatletter
\renewcommand*\env@matrix[1][*\c@MaxMatrixCols c]{%
  \hskip -\arraycolsep
  \let\@ifnextchar\new@ifnextchar
  \array{#1}}
\makeatother

\DeclareMathOperator{\tr}{tr}
\DeclareMathOperator{\JSR}{JSR}

\newcommand{\tribar}[1]{\mathopen{| {\kern -1.5pt} | {\kern -1.5pt} |} {#1}
\mathclose{| {\kern -1.5pt} | {\kern -1.5pt} |}}

\newcommand{\arxiv}[1]{\href{http://arxiv.org/abs/#1}{arXiv:{#1}}}
\newcommand{\doi}[1]{\href{http://dx.doi.org/#1}{\tt DOI}} %{DOI:{#1}}}
\newcommand{\directlink}[1]{\href{#1}{\tt URL}}
\newcommand{\MRev}[1]{\href{https://mathscinet.ams.org/mathscinet-getitem?mr=#1}{\tt MR}} %{MR{#1}}}
\newcommand{\Zbl}[1]{\href{https://zbmath.org/?q=an:#1}{\tt Zbl}} %{Zbl{#1}}}
\newcommand{\OEIS}[1]{\href{https://oeis.org/#1}{\tt{#1}}}

\begin{document}

\title{Spectrum Maximizing Products are not generically unique}

\author{Jairo Bochi and Piotr Laskawiec}
\address{Department of Mathematics, The Pennsylvania State University}
\email{\href{mailto:bochi@psu.edu}{bochi@psu.edu}, \href{mailto:ppl5146@psu.edu}{ppl5146@psu.edu}}

\date{First version: January, 2023. Revision: August, 2023.}


\begin{abstract}
It is widely believed that typical finite families of $d \times d$ matrices admit finite products that attain the joint spectral radius. This conjecture is supported by computational experiments and it naturally leads to the following question: are these spectrum maximizing products typically unique, up to cyclic permutations and powers? We answer this question negatively. As discovered by Horowitz around fifty years ago, there are products of matrices that always have the same spectral radius despite not being cyclic permutations of one another. We show that the simplest Horowitz products can be spectrum maximizing in a robust way; more precisely, we exhibit a small but nonempty open subset of pairs of $2 \times 2$ matrices  $(A,B)$ for which the products $A^2 B A B^2$ and $B^2 A B A^2$ are both spectrum maximizing. 
\end{abstract}

\maketitle

\section{Introduction}

The \emph{joint spectral radius} of a family of linear operators was introduced by Rota and Strang \cite{rota-strang} in 1960 and later became a topic of intense research. It measures the maximal asymptotic growth rate of products of matrices drawn from the family; more precisely: 
\begin{definition}
Let $\mathcal{A}$ be a bounded set of real or complex $d\times d$ matrices.
The \emph{joint spectral radius} (\emph{JSR} in short) of $\mathcal{A}$ is given by: 
\begin{equation}\label{e.def_JSR}
\JSR(\mathcal{A}) \coloneqq \lim_{k\to\infty}\sup_{\Pi \in \mathcal{A}^{*k}}\|\Pi\|^{\frac{1}{k}} \, ,
\end{equation}
where $\mathcal{A}^{*k}$
denotes the collection of all products of elements of~$\mathcal{A}$ with length $k$.
\end{definition}
The limit \eqref{e.def_JSR} exists as a consequence of Fekete's lemma and is independent of the matrix norm $\| \mathord{\cdot} \|$. 
In the case that $\mathcal{A}$ is a singleton $\{A\}$,
Gelfand's formula tells us that
the JSR coincides with the spectral radius:
\begin{equation} \label{e.def_rho}
\rho(A) \coloneqq \max\{|\lambda| : \lambda \text{ is an eigenvalue of }A \} \, ,
\end{equation}
and so it is easy to compute.
The situation is completely different   for two or more matrices.
The JSR is no longer a semialgebraic function: see \cite{nonalgebraic}.
It was shown in \cite{hard} that the JSR of a pair of integer matrices cannot be computed in polynomial time, unless $P=\mathit{NP}$.
Worse than that,
among pairs of rational matrices, it is an undecidable problem whether $\JSR\leq 1$, in the sense that no algorithm is able to reliably answer such question in a finite amount of time: see~\cite{undecidable}. 

Still, computable bounds for the JSR are available. 
As noted in \cite[Lemma~3.1]{DaubechiesL}, the JSR of a bounded family $\mathcal{A}$ can be squeezed between two sequences:
\begin{equation}\label{e.squeeze}
\sup_{\Pi \in \mathcal{A}^{*k}}\rho(\Pi)^\frac{1}{k}\leq \JSR(\mathcal{A})\leq \sup_{\Pi \in \mathcal{A}^{*k}}\|\Pi\|^{\frac{1}{k}} \, .
\end{equation}
The upper bound converges to the JSR as $k \to \infty$, by its very definition \eqref{e.def_JSR}.
As for the lower bound, convergence may fail, but the situation can be remedied by taking a limsup:
\begin{equation}\label{e.BW}
\JSR(\mathcal{A}) = \limsup_{k\to\infty}\sup_{\Pi \in \mathcal{A}^{*k}}\rho(\Pi)^\frac{1}{k} \, ; 
\end{equation}
this is the Berger--Wang theorem \cite{berger-wang} (see \cite[Theorem~4]{CGP}, \cite{Breuillard} for an improved version).
Therefore one can use the inequalities \eqref{e.squeeze} to approximate the JSR: see e.g.\ \cite[\S~2.3.3]{Jungers} for discussion of some computational issues.

From now on, let us restrict the discussion to finite families $\mathcal{A} = \{A_1,\dots,A_n\}$, in which case we write $\JSR(\mathcal{A}) = \JSR(A_1,\dots,A_n)$.
In favorable circumstances, the bounds in \eqref{e.squeeze} may become equalities, and then the computation of the JSR is greatly simplified. 
Under mild conditions on the set of matrices $\mathcal{A}$, it admits an \emph{extremal norm}, that is, a norm on $\mathbb{R}^d$ for which the second inequality in \eqref{e.squeeze} becomes an equality with $k=1$: see \cite{nonalgebraic} or \cite[\S~2.1.2]{Jungers}.

The situation where the first inequality in \eqref{e.squeeze} becomes an equality is central to this note.
It is convenient to focus our attention on products $\Pi = A_{i_1} \cdots A_{i_k}$  that are \emph{primitive}, that is, the word $i_1\cdots i_k$ is not a power of a shorter word. (This definition should not be confused with the notion of primitivity used in group theory.)

\begin{definition}
Let $\Pi\in \mathcal{A}^{*k}$ be a primitive product of lenght $k$.
We say that $\Pi$ is a \emph{spectrum maximizing product} (\emph{SMP} in short) if
\begin{equation}
\JSR(\mathcal{A})=\rho(\Pi)^\frac{1}{k} \, .
\end{equation}
\end{definition}
At an early stage, it was hoped  that  every finite family of matrices would have an SMP \cite{Wang}. This came to be known as the \emph{finiteness conjecture} which was eventually refuted by Bousch and Mairesse \cite{finiteness}; later constructions were offered by several authors \cite{BTV,Kozyakin,Hare, JenkinsonPollicott}.
On the other hand, the counterexamples seem to form a small set, and 
Maesumi \cite[Conjecture~8]{Maesumi} conjectured that  Lebesgue-almost every finite family of matrices has an SMP. 
The conjecture seems to be confirmed by experiments: see  \cite{algo1,Mejstrik}.

One way of certifying that a certain product $\Pi \in \mathcal{A}^{*k}$ is spectrum maximizing for a given family $\mathcal{A}$ is to find a norm $\|\mathord{\cdot}\|$ for which $\rho(\Pi)^{\frac{1}{k}} = \sup_{A\in \mathcal{A}} \|A\|$ and then apply \eqref{e.squeeze} to conclude simultaneously  that $\Pi$ is a SMP and $\|\mathord{\cdot}\|$ is an extremal norm. 
Thus, the problem boils down to efficiently finding such extremal norms.  A solution is provided by the \emph{polytope algorithm} \cite{GWZ2005,algo1,algo2}. See also \cite{VGJ,Protasov21,Mejstrik} for further information.

When it terminates, the polytope algorithm not only validates a product as an SMP, but it also ensures it is \emph{robust} in the sense that if the matrices in the (finite) family $\mathcal{A}$ are slightly perturbed, the same product is still spectrum maximizing. So it seems plausible that the counterexamples to the finiteness conjecture form a \emph{nowhere dense} set of zero Lebesgue measure. 

Cyclic permutations of an SMP are also SMPs. 
Ignoring such repetitions, experimentally found SMPs are usually unique. 
This raises the following question: 
\begin{question}\label{q.unique}
Does Lebesgue-almost every finite family of matrices have a unique SMP, up to cyclic permutations?
\end{question}
More precisely, given integers $d\ge 2$ and $n\ge 2$, does the set of $n$-tuples of $d\times d$ matrices with the unique SMP property have full Lebesgue measure in $\mathbb{R}^{nd^2}$?

This sort of problem pertains to the wider field of \emph{ergodic optimization}. The JSR of a bounded family of matrices $\mathcal{A}$ is the maximal Lyapunov exponent of a related linear cocycle over a shift dynamics: see \cite{BochiICM} and references therein. In this context, there always exist at least one \emph{Lyapunov-maximizing measure}, meaning a shift-invariant probability measure with  maximal Lyapunov exponent, and if such a measure is supported on a periodic orbit then the family of matrices admits an SMP.
Furthermore, cyclic permutations of an SMP correspond to exactly the same measure. Therefore \cref{q.unique} is asking about typical existence of a unique finitely-supported Lyapunov-maximizing measure.

Ergodic optimization is more developed in the commutative situation where one seeks to maximize Birkhoff averages instead of Lyapunov exponents. In that situation, Hunt and Ott  \cite{HuntOtt} conjectured that having a unique maximizing measure supported on a periodic orbit should be a typical property, at least for sufficiently chaotic underlying dynamical systems. Contreras \cite{Contreras} has proved a positive result in that direction, using however a topological notion of typicality. On the other hand,  uniqueness of maximizing measures (finitely supported or not)  was proved to be a typical property in topological and probabilistic senses by Jenkinson \cite[Theorem~3.2]{Jenk} and Morris \cite{Morris_prevalent}, respectively.

In this state of affairs, it came as a surprise to us that \cref{q.unique} actually has a negative answer. 
Our starting point is the following curious algebraic observation:

\begin{lemma}\label{l.SC}
For all pairs of $2 \times 2$ (real or complex) matrices $A$, $B$, the two matrix products
\begin{equation}\label{e.SC_pair}
A^2BAB^2 \quad \text{and} \quad B^2ABA^2
\end{equation}
have the same eigenvalues. 
\end{lemma}
Note that this \lcnamecref{l.SC} is not immediate, since the  words \eqref{e.SC_pair} are not cyclic permutations of one another; actually, for $3 \times 3$ matrices, it is not true that $\rho(A^2BAB^2) \equiv \rho(B^2ABA^2)$. \cref{l.SC} is the simplest manifestation of a more general phenomenon uncovered by Horowitz about fifty years ago; we will elaborate on this in \cref{s.algebra}. 
The main result of this note is that the products \eqref{e.SC_pair} can be robust SMPs, and in particular \cref{q.unique} has a negative answer.
More precisely:

\begin{theorem}\label{t.main}
There exists a nonempty open set $\mathcal{U}$ in $\mathcal{M}_2(\mathbb{R})^2$ such that for all $(A,B) \in \mathcal{U}$, the products $A^2BAB^2$ and $B^2ABA^2$ are spectrum maximizing. 
Furthermore, there are no other SMPs other than their cyclic permutations.
\end{theorem}

A similar statement holds in the setting of complex matrices: see \cref{t.complex} below.

Our results are also relevant for the study of a special kind of extremal norms called \emph{Barabanov norms}. We recall the definition:

\begin{definition}
A vector norm $\tribar{\mathord{\cdot}}$ on $\mathbb{R}^d$ is a \emph{Barabanov norm} for a bounded family $\mathcal{A} \subseteq \mathcal{M}_d(\mathbb{R})$ if 
\begin{equation}\label{e.Barabanov}
\forall v \in \mathbb{R}^d, \quad \sup_{A\in \mathcal{A}} \tribar{Av} = \JSR(\mathcal{A}) \, \tribar{v} \, , 
\end{equation}
\end{definition}

Barabanov norms can be computed using polytope algorithms: see \cite{GZ15}. They are often unique (modulo rescaling). In fact, uniqueness of the SMP together with a dominance condition guarantees uniqueness of the Barabanov norm: see
 \cite[Theorem~2]{Protasov21}. On the other hand, we show:

\begin{corollary}\label{c.Barabanov}
The set of pairs of matrices that have a unique Barabanov norm (modulo rescaling) is \emph{not} dense in $\mathcal{M}_2(\mathbb{R})^2$.
\end{corollary}

Our findings are strictly two-dimensional. In fact,  \cref{l.SC} fails in higher dimensions, and there is no known analogous phenomenon. See \cref{ss.further,s.questions} for further discussion.

%%%%%%%%%%%%%%%%%%%%%%%%%%%%%%%%%%%%%%%%%%%%%%%%%%%%%%
\section{Proof of the main theorem}\label{s.example}

In this section, we  prove \cref{t.main} by exhibiting an explicit example.

\subsection{Proof of Lemma~\ref{l.SC}}

As already mentioned, this \lcnamecref{l.SC} is known: indeed, it follows from \cite[eq.~(5.2b)]{Horowitz}. It was rediscovered at least once: see \cite[p.12]{Selinger}. For the convenience of the reader, let us reproduce the proof from~\cite{ME}.

\begin{proof}[Proof of the \lcnamecref{l.SC}]
Given $2\times 2$ matrices $A$, $B$, consider $C \coloneqq AB-BA$.
Since $\tr C = 0$, Cayley--Hamilton theorem tells us that $C^2$ is a scalar multiple of identity (possibly zero). 
In particular, $\tr C^3 = \tr C = 0$.
On the other hand,
\begin{multline}
C^3 = 
(ABABAB - BABABA)
+ (-ABABBA-ABBAAB-BAABAB) \\ 
+ (ABBABA+BAABBA+BABAAB) \, ,
\end{multline}
and the products in each set of parentheses are cyclic permutations of one another.
Taking the trace we obtain
\begin{equation}
- 3 \tr(A^2BAB^2) 
+ 3 \tr(B^2ABA^2) = 0  \, , 
\end{equation}
so $A^2BAB^2$ and $B^2ABA^2$ have the same trace. 
Since these two $2 \times 2$ matrices also have the same determinant, they have the same eigenvalues. 
\end{proof}

In fact, \cref{l.SC} is just the simplest instance of a more general phenomenon, which we  explain in \cref{s.algebra}.

\subsection{The example}

\ The procedure used in this section to certify SMPs is a specific instance of the invariant polytope algorithm with balancing procedure \cite{algo2,GZ15}. On the other hand, the following proof is self-contained and does not require knowledge of the polytope algorithm. The arguments involved are not new; the novelty here is the example itself. The proof ultimately relies on finitely many algebraic inequalities, which can be checked numerically on any computer algebra system.

We begin by defining the following pair of matrices:
\begin{equation}\label{e.the_example}
A_0 \coloneqq \begin{bmatrix}
0.81427& -0.32898\\
0.73419& \phantom{-}0.50393
\end{bmatrix}
\quad\text{and}\quad
B_0 \coloneqq \begin{bmatrix}
-0.06078& \phantom{-}1.01008\\
-0.88368& -0.26830
\end{bmatrix}
\end{equation}
The dominant eigenvalue of the product $A_0^2B_0A_0B_0^2$ is approximately $-0.99998$.
Then, we define two new matrices: 
\begin{equation}\label{e.matrices}
A \coloneqq \frac{A_0}{\rho(A_0^2B_0A_0B_0^2)^{1/6}} \quad\text{and}\quad B \coloneqq \frac{B_0}{\rho(A_0^2B_0A_0B_0^2)^{1/6}} \, .
\end{equation}
Actually $A$ and $B$ coincides with $A_0$ and $B_0$ up to $5$ digits.
It follows from \cref{l.SC} that the matrix products \eqref{e.SC_pair} (as well as their cyclic permutations) have dominant eigenvalue exactly $-1$. 
Next, we consider the following eigenvectors
\begin{equation}\label{e.eigenvectors}
v_4\approx\begin{bmatrix}
0.63620\\ 0.77152
\end{bmatrix} \quad\text{and}\quad v_{9}\approx\begin{bmatrix}
0.88452\\ 0.02929
\end{bmatrix}
\end{equation}
of the matrices $A^2BAB^2$ and $B^2ABA^2$, respectively, both associated to the leading eigenvalue $-1$, with Euclidean lengths exactly $1$ and $0.885$, respectively.  (Let us remark that these lengths are not arbitrary; they come from the balancing procedure, which is described in detail in \cite{algo2}.)

Then, we compute 14 other vectors by taking successive images of $v_4$ and $v_{9}$ under $\pm A$ and $\pm B$, as specified by the (disconnected) graph below.
\begin{equation}\label{e.graph}
\begin{tikzpicture}[scale=1.15,baseline=(current  bounding  box.center)]
\tikzset{vertex/.style = {shape=circle,minimum size=0em}}
\tikzset{edge/.style = {->,> = latex'}}
	
\draw (0,0) circle (.2cm);
\node[vertex] (za) at  (0,0) {$v_{4}$};
\node[vertex] (zb) at  (1,1) {$v_{13}$};
\node[vertex] (ze) at  (1+1.4,1) {$v_{6}$};
\node[vertex] (zf) at  (1+2.4,0) {$v_{2}$};
\node[vertex] (zg) at  (1+1.4,-1) {$v_{11}$};
\node[vertex] (zh) at  (1,-1) {$v_{8}$};
\node[vertex] (zee) at  (1+2.4,2) {$v_{15}$};
\draw[edge] (za) to["$B$"] (zb);
\draw[edge] (zb) to["$-B$"]  (ze);
\draw[edge] (ze) to["$A$"]  (zf);
\draw[edge] (zf) to["$B$"] (zg);
\draw[edge] (zg) to["$A$"]  (zh);
\draw[edge] (zh) to["$-A$"] (za);
\draw[edge] (ze) to["$B$"] (zee);
\begin{scope}[xshift=4.5cm]
 \draw (0,0) circle (.25cm);
 \node[vertex] (a) at  (0,0) {$v_{9}$};
  \node[vertex] (b) at  (1,1) {$v_{5}$};
\node[vertex] (e) at  (1+1.4,1) {$v_{1}$};
\node[vertex] (f) at  (1+2.4,0) {$v_{10}$};
\node[vertex] (g) at  (1+1.4,-1) {$v_{7}$};
\node[vertex] (h) at  (1,-1) {$v_{16}$};
\node[vertex] (i) at  (1+2.4,-2) {$v_{3}$};
\node[vertex] (j) at  (1+2.4+1.4,-2) {$v_{12}$};
\node[vertex] (k) at  (0,2) {$v_{14}$}; 
\draw[edge] (a) to["$A$"] (b);
\draw[edge] (b) to["$A$"]  (e);
\draw[edge] (e) to["$B$"]  (f);
\draw[edge] (f) to["$A$"] (g);
\draw[edge] (g) to["$B$"]  (h);
\draw[edge] (h) to["$B$"] (a);
\draw[edge] (g) to["$A$"]  (i);
\draw[edge] (g) to["$A$"]  (i);
\draw[edge] (i) to["$B$"]  (j);
\draw[edge] (b) to["$B$"]  (k);
\end{scope}
\end{tikzpicture}
\end{equation}

Lastly, we enlarge our list by including all opposite vectors $v_{17}\coloneqq -v_1,\dots,v_{32}\coloneqq -v_{16}$.

\begin{lemma}\label{l.convex}
The points $v_1,v_2,\dots,v_{36}$ are the cyclically ordered vertices of a convex
$36$-gon $S$.
\end{lemma}
    The lemma is easily verified using a computer algebra system. The polygon $S$ is shown in \cref{f.polygon}. The largest interior angle is approximately $ 175.8 ^\circ$ at the vertices $v_{7},v_{23}$.

\begin{figure}[ht]
\centering
\includegraphics[width=\textwidth]{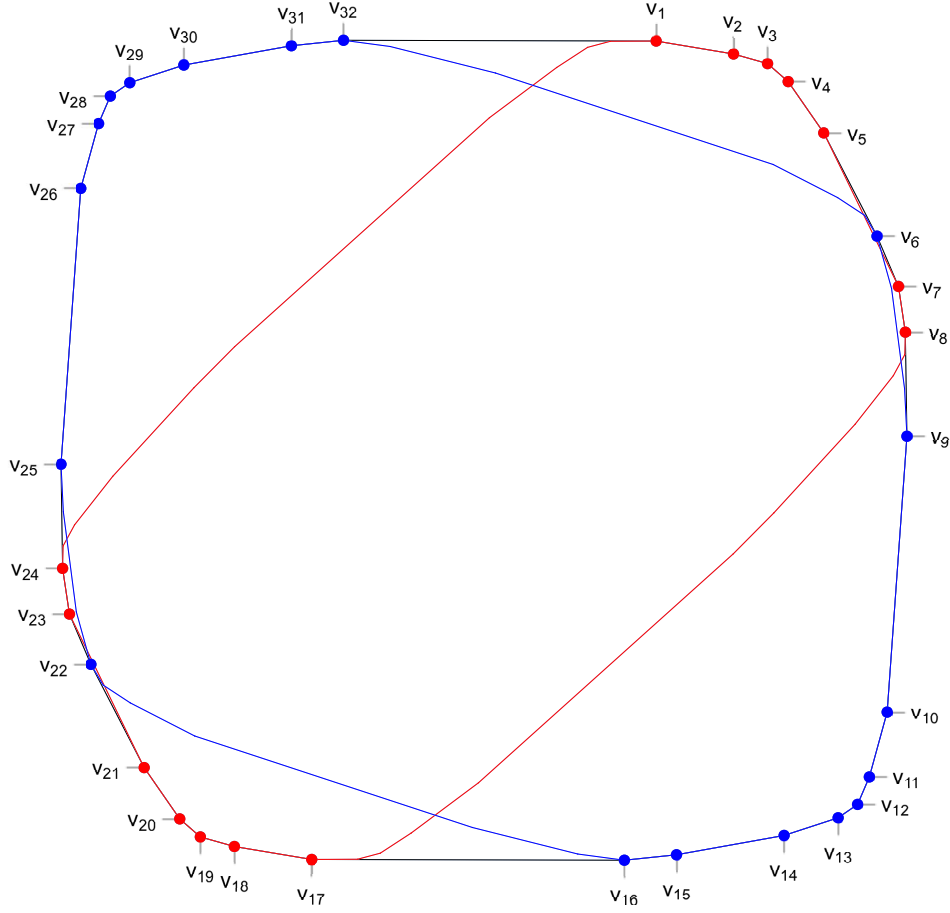}
\caption{The polygon $S$ and its images $A(S)$ (red) and $B(S)$ (blue).}\label{f.polygon}
\end{figure}

\begin{lemma}\label{l.images}
Given a vertex $v_i$, its image $Av_i$ (resp.\ $Bv_i$) either is another vertex $v_j$ or it belongs to the interior of $S$.
\end{lemma}
This lemma can also be verified using a computer. The closest call is the image $Av_{4}$ (or its opposite), which is approximately $7.6*10^{-4}$
away from the boundary of $S$.

The lemma implies that $AS\subseteq S$ and $BS\subseteq S$. 
In particular, any product $\Pi$ of $A$'s and $B$'s satisfies $\Pi(S) \subseteq S$, and therefore the norm $\| \Pi \|$ is bounded by a constant. That is, the pair $\{A,B\}$ is \emph{product bounded}. Hence, $\JSR(A,B) \le 1$.
On the other hand, $\JSR(A,B) \ge \rho(A^2BAB^2)=1$. Therefore, $\JSR(A,B)=1$ and so the products \eqref{e.SC_pair} are SMPs.
The next step is to prove that they are (essentially) unique.

\begin{lemma}\label{l.infinite_word}
Consider a sequence of matrices $(A_k)_{k\ge 1}$ composed of $A$'s and $B$'s, and let $\Pi_k \coloneqq A_k \cdots A_1$. Then:
\begin{enumerate}[label={\normalfont({\alph*})},ref={\normalfont{\alph*}}]
\item\label{i.contract} 
either $\Pi_k(S)$ is contained in the interior of $S$ for some $k$,
\item\label{i.periodic} 
or the sequence $(A_k)$ is periodic of period $6$, and $\Pi_6$ is a cyclic permutation of either $A^2 BA B^2$ or $B^2 AB A^2$.
\end{enumerate}
\end{lemma}

\begin{proof}
Assume that alternative \eqref{i.contract} does not hold, that is, for all $k$, the image $\Pi_k(S)$ is not contained in the interior of $S$. It follows from \cref{l.images} that for each $k \ge 1$ there must be at least one vertex $v_{i_k}$ that maps to another vertex $v_{j_k}$, i.e.,
\begin{equation}
\Pi_k(v_{i_k})=v_{j_k} \, .
\end{equation}
By the pigeonhole principle, some vertex $v_{m_0}$ appears infinitely often in the sequence $(v_{i_k})_{k\ge 1}$. Recall from  \cref{l.images} that once a vertex is mapped to an interior point, it must stay in it under any product of $A$ and $B$. Hence for all $k$, $\Pi_k(v_{m_0})$ must be another vertex of $S$, say,
\begin{equation}
\Pi_k(v_{m_0})=v_{m_k} \, .
\end{equation}
Equivalently, $v_{m_k} = A_k (v_{m_{k-1}})$ for all $k \ge 1$.
Replacing some vertices by their opposites if necessary, what we obtain is an infinite path in the graph \eqref{e.graph}. However, every such path consists of winding around one of the two cycles. It follows that alternative \eqref{i.periodic} in the \lcnamecref{l.infinite_word} holds.
\end{proof}

The lemma has the following simple corollary:
\begin{lemma}\label{l.unique}
Let $\Pi$ be a  primitive product of $A$'s and $B$'s that is not a cyclic permutation of either $A^2 BA B^2$ or $B^2 AB A^2$. Then, $\Pi^k(S)$ is contained in the interior of $S$ for some $k$.
\end{lemma}

\begin{proof}
Given a primitive word $\Pi$ of length $\ell$, write it as $\Pi = A_\ell \cdots A_1$ where each $A_i$ is either $A$ or $B$. Now extend $(A_k)_{k=1}^\ell$ to a periodic sequence $(A_k)_{k\ge 1}$ of period $\ell$, and apply \cref{l.infinite_word} to this sequence.
If case \eqref{i.contract} holds, then we have nothing to prove.
So, assume we we are in case \eqref{i.periodic}, that is,  the sequence  $(A_k)_{k\ge 1}$ also has period $6$, and $C \coloneqq A_6 \cdots A_1$ is a cyclic permutation of either $A^2 BA B^2$ or $B^2 AB A^2$. In particular, $\Pi^6 = C^\ell$.
Since both $\Pi$ and $C$ are primitive, it follows from uniqueness of degrees and roots (see \cite{roots}) that 
$\ell = 6$ and $\Pi = C$.
\end{proof}

\cref{l.unique} shows that any primitive $\Pi$ that is not a cyclic permutation of one of the two products \eqref{e.SC_pair} 
has $\rho(\Pi)<1$ and so it cannot be a SMP for the family $\{A,B\}$.
This proves uniqueness. 
So, coming back to the original pair $(A_0,B_0)$ defined in \eqref{e.the_example}, we see that it has SMPs \eqref{e.SC_pair}, and no other SMPs other than their cyclic permutations. The same conclusions remain valid if the matrices are slightly perturbed. Indeed, the vertices $v_i$ will move continuously, so \cref{l.convex,l.images} still hold. Furthermore, \cref{l.infinite_word,l.unique} are just abstract consequences, so they also persist. 
The proof of \cref{t.main} is concluded.

\subsection{Additional information}\label{ss.additional}

\cref{f.polygon} indicates that our example is delicate. 
In fact, the normalized spectral radii of the product $A^3 B A^2 B$ is $0.99936$, and it turns out that if we add $0.005$ to the $(2,1)$ entry of the matrix $B$, then this product becomes the unique SMP (modulo cyclic permutations, of course).
The set $\mathcal{U}$ we have found in our proof of \cref{t.main} seems to have a quite small inradius.

Let us mention a few direct extensions of our result:
\begin{itemize}
\item  \cref{t.main} obviously holds for $n$-tuples $(A_1,\dots,A_n)$ of $2 \times 2$ real matrices, for any $n \ge 2$; 
indeed, if the other $n-2$ matrices are sufficiently small, then they do not affect the SMPs.  More precisely, if $A_1=A$, $A_2=B$, and all the other matrices satisfy the following inequality:
\begin{equation}
\|A_i\|_S< \JSR(A,B) \quad (3\le i \le n),
\end{equation}
where $\|\mathord{\cdot} \|_S$ is the norm induced by the polygon $S$, then the SMPs remain unchanged.

\item For the matrices $A,B$ we defined, the pair of products \eqref{e.SC_pair} 
is actually \emph{dominant} in the sense of Guglielmi and Protasov \cite[p.~25]{algo2}.

\item It follows directly from \cref{l.infinite_word} that the \emph{Mather set} (see  Morris \cite{Morris_Mather}) of our pair of matrices is the union of two periodic orbits. This means that there are no other maximizing measures apart from those two coming from the periodic orbits corresponding to \eqref{e.SC_pair}.
\end{itemize}

\medskip

Now let us see why our construction implies non-uniqueness of Barabanov norms:

\begin{proof}[Proof of \cref{c.Barabanov}]
By the theorem of Dranishnikov--Konyagin--Protasov \cite{Protasov96}, every irreducible family  $\mathcal{A}$ with unit JSR admits a centrally symmetric convex body $C$ in $\mathbb{R}^d$ with the property that the closed convex hull of $\bigcup_{A \in \mathcal{A}} A(C)$ equals $C$.
Specializing to the pair $(A,B) \in \mathcal{M}_2(\mathbb{R})^2$ defined in \eqref{e.matrices}, the polygon $S$ introduced before has this property.
It follows from a duality result of Plischke and Wirth \cite{PW} that the norm $\tribar{\mathord{\cdot}}$ in $\mathbb{R}^2$ whose unit ball is the polar of the polygon $S$ is a Barabanov norm for the pair of transpose matrices $(A^\mathtt{t},B^\mathtt{t})$.
Recall that the eigenvectors \eqref{e.eigenvectors} were chosen so that the ratio of their lengths is exactly $0.885$ (which is computed using the balancing algorithm, and therefore is optimal in the sense explained in \cite{algo2}). However, any sufficiently close ratio would have worked, leading to a slightly different (non-similar) polygon $S$, and therefore to a non-homothetic Barabanov norm for  $(A^\mathtt{t},B^\mathtt{t})$. As explained before, if the matrices are perturbed, the polygons can be adjusted accordingly. So non-uniqueness of Barabanov norms holds on an open set of pairs.
\end{proof}

\subsection{Complex matrices}

Our main theorem holds true for complex matrices as well:

\begin{theorem}\label{t.complex}                          
There exists a nonempty open set $\mathcal{U}$ in $\mathcal{M}_2(\mathbb{C})^2$ such that for all $(A,B) \in \mathcal{U}$, the products $A^2BAB^2$ and $B^2ABA^2$ are spectrum maximizing. Furthermore, there are no other SMPs other than their cyclic permutations.
\end{theorem}

Before giving the proof, we need to recall a few notions (see \cite{absco} for further information).

\begin{definition}
A set $S\subseteq \mathbb{C}^2$ is \emph{absolutely convex} if it is convex and it is invariant under multiplication by complex numbers of modulus $1$.
\end{definition}

\begin{definition}
The \emph{absolute convex hull} of a finite set $F\subseteq \mathbb{C}^2$, denoted $\operatorname{absco}(F)$, is the smallest absolutely convex set that contains $F$.
\end{definition}
In other words, the absolute convex hull of a set is the intersection of all absolutely convex sets that contain it. 
It can be also be expressed as:
\begin{equation}\label{e.absco}
\operatorname{absco}(S)=\left\{\sum_{i=1}^{n} \lambda_iv_i :  v_1,\dots,v_n \in S, \ \lambda_1, \dots,\lambda_n \in \mathbb{C}, \ \sum_{i=1}^{n} |\lambda_i| \leq 1\right\} \, .
\end{equation}

\begin{proof}[Proof of \cref{t.complex}]
We follow the same strategy as in the proof of \cref{t.main} but with slight modifications.
Let $A$ and $B$ be our real matrices \eqref{e.matrices} and let $\widetilde{A}$ and $\widetilde{B}$ be complex perturbations. Without loss of generality, assume that these perturbations are normalized so that $\rho(\widetilde{A}^2\widetilde{B}\widetilde{A}\widetilde{B}^2)=1$. We can proceed with the same process of creating 32 vectors $\tilde{v}_i,\dots, \tilde{v}_{32}$, following the same graph \eqref{e.graph}, just as in the real case. Let $T$ be the absolute convex hull of the points $\tilde{v}_i$.
Note that the starting vectors $\tilde{v}_4$ and $\tilde{v}_9$
are not uniquely determined. However, the indeterminacy is given by multiplication by complex numbers of modulus $1$, and therefore does not influence the set $T$.

We claim that $\widetilde{A}T\subseteq T$ and $\widetilde{B}T\subseteq T$. By linearity, it is sufficient to show that $\widetilde{A}\tilde{v}_j\in T$ and $\widetilde{B}\tilde{v}_j\in T$ for all $j$. 
We will consider images under $\tilde A$ only, as the case of $\tilde B$ is entirely analogous.

 Recall from \cref{l.images} that there exist no vectors of the form $Av_j$ or $Bv_j$ (with $v_j$ being a vertex of the real polygon $S$) lying on the boundary of the polygon without being a vertex $v_i$. 
If $A v_j = v_i$, then we must have $\widetilde{A} \tilde{v}_j = \tilde{v}_i$, since the perturbed vectors are defined using the same graph. In particular, $\widetilde{A} \tilde{v}_j$ belongs to the set $T$, as we wanted to show.  

If $Av_j$ is not a vertex of $S$, then $Av_j\in \operatorname{int}(S)$. We can slice our convex centrally symmetric polygon $S$ into triangles $T_k$ with vertices $(0,0),v_k,v_{k+1}$, where indices are meant modulo $32$. Fix a slice $T_k$ containing $Av_j$. Then there are unique numbers $\alpha \ge 0$, $\beta \ge 0$ such that 
\begin{equation}
Av_j = \alpha v_{k} +\beta v_{k+1}
\end{equation}
and $\alpha+\beta<1$.
As long as the perturbations are small enough, we know there must also be a unique real solution $(\widetilde{\alpha},\widetilde{\beta})$ to the linear system:
\begin{equation}
\widetilde{A}\tilde{v}_j=\widetilde{\alpha}\tilde{v}_{k}+\widetilde{\beta}\tilde{v}_{k+1} \, ,
\end{equation}
and this solution will satisfy $|\widetilde{\alpha}|+|\widetilde{\beta}|<1$.
 By \eqref{e.absco}, $\widetilde{A}\tilde{v}_j\in T$, as we wanted to prove.

We have shown that $\widetilde{A}T\subseteq T$ and $\widetilde{B}T\subseteq T$. It follows that $\JSR(\widetilde{A},\widetilde{B})\leq 1$. Since $\rho(\widetilde{A}^2\widetilde{B}\widetilde{A}\widetilde{B}^2)=1$, we have, just as in the real case, $\JSR(\widetilde{A},\widetilde{B})=1$, as desired.
\end{proof}

\section{On the coincidence of eigenvalues}\label{s.algebra}

In this section, we will see that \cref{l.SC} is not an algebraic accident. 

\subsection{Fricke polynomials}
 
Let $\mathbb{F}^+_2$ 
denote the free semigroup on generators $a$, $b$, that is, the set of all finite words in the letters $a$, $b$ 
with the operation of concatenation. 
Let $\mathcal{M}_d(\mathbb{C})$ denote the semigroup of complex $d \times d$ matrices. 
If $w \in \mathbb{F}^+_2$ and 
$A,B \in \mathcal{M}_d(\mathbb{C})$, 
then we denote by $w(A,B)$ the matrix product obtained by replacing the letters $a$, $b$ in the word $w$  by the matrices $A$, $B$. That is, $w \mapsto w(A,B)$ is the unique homomorphism from $\mathbb{F}^+_2$ to $\mathcal{M}_d(\mathbb{C})$ that maps $a$ to $A$ and $b$ to $B$.

\begin{proposition}\label{p.Fricke}
For every $w \in \mathbb{F}^+_2$, there exists a unique polynomial $F_w$ in five variables such that, for all $2 \times 2$ matrices $A$, $B$, we have
\begin{equation}\label{e.Fricke_id}
\tr w(A,B) = F_w(\tr A,\tr B,\tr AB,\det A,\det B) \, .
\end{equation}
Furthermore, $F_w$ has integer coefficients.
\end{proposition}

This fact 
(or actually, its more general version for the free \emph{group} $\mathbb{F}_2$)
is known since the 19th century: see \cite{Horowitz,Magnus,Goldman}.
Nevertheless, for the convenience of the reader, let us give a proof, essentially following an observation of Sylvester~\cite{Sylvester}:

\begin{proof}
Let us 
write 
\begin{equation}\label{e.5_var}
x = \tr A, \quad
y = \tr B, \quad
z = \tr AB, \quad
u = \det A,\quad 
v = \det B. 
\end{equation}
By the Cayley-Hamilton (CH) theorem,
\begin{equation}
A^2 - xA + uI = 0, \quad
B^2 - yB + vI = 0.
\end{equation}
Also, by the multilinear CH theorem \cite[eq.~(4)]{Formanek} (itself a direct consequence of the regular CH theorem),
\begin{equation}
AB + BA - yA - xB + (xy-z)I = 0 .
\end{equation}
By using these three equations repeatedly, we can express any product $w(A,B)$ of $A$'s and $B$'s  as a linear combination of the four matrices $I$, $A$, $B$, $AB$, with the coefficients being polynomials with integer coefficients in the variables $x$, $y$, $z$, $u$, $v$.
Since the trace is linear, existence  follows.

We now prove uniqueness. Given any point $(x,y,z,u,v)$ in $\mathbb{C}^5$, we can  find matrices $A$ and $B \in \mathcal{M}_2(\mathbb{C})$ such that \eqref{e.5_var} holds; for example, an easy calculation shows that we can take $A$ to be upper-triangular and $B$ lower-triangular. 
Now, given a word $w \in \mathbb{F}^+_2$, we can use identity \eqref{e.Fricke_id} to compute the value $F_w(x,y,z,u,v)$. This shows that the the polynomial function $F_w$ is unique.
\end{proof}

We call $F_w(x,y,z,u,w)$ a \emph{Fricke polynomial}. When determinants are set to $1$, we get the \emph{reduced Fricke polynomials} 
$f_w(x,y,z) \coloneqq F_w(x,y,z,1,1)$, which are actually more common in the literature. Let us note that $F_w$ is uniquely determined by $f_w$.
Indeed, since the trace and determinant are homogeneous functions on $\mathcal{M}_2(\mathbb{C})$ of respective degrees $1$ and $2$, it follows that every polynomial $F_w(x,y,z,u,v)$ is weighted homogeneous with weights 
$(1,0,1,2,0)$ and also with weights
$(0,1,1,0,2)$.

While the proof of \cref{p.Fricke} suggests a method for the computation of the Fricke polynomials, it is not the most efficient one. See the references mentioned above for better algorithms; see also \cite{Jorgensen} for a general formula and \cite{Chas} for an online calculator of Fricke polynomials.

\subsection{Isospectral products and chiral pairs}\label{ss.isospectral}

\begin{definition}\label{isospectral}
Let $d\ge 2$ be an integer. We say that two words $w_1,w_2$ in $\mathbb{F}^+_2$ are $d$-\emph{isospectral} if for all $A$, $B \in \mathcal{M}_d(\mathbb{C})$, the matrices $w_1(A,B)$ and $w_2(A,B)$ have the same eigenvalues, with the same multiplicities.
\end{definition}

Note that any word is $d$-isospectral to its cyclic permutations.

Specializing to dimension $d=2$, we have the following characterization:

\begin{proposition}\label{p.2-isospec}
Two words  $w_1,w_2 \in \mathbb{F}^+_2$ are $2$-isospectral if and only if they have identical Fricke polynomials: $F_{w_1} \equiv F_{w_2}$. In this case, $w_2$ must be a permutation of $w_1$.
\end{proposition}

\begin{proof}
The forward direction is clear as trace is the sum of the eigenvalues.
Now assume that the words $w_1$, $w_2$ have identical Fricke polynomials, that is, $\tr{w_1(A,B)}=\tr{w_2(A,B)}$ for all $A$, $B \in \mathcal{M}_2(\mathbb{C})$. Let $n_1$ (resp.\ $n_2$) be the number of letters $a$ in $w_1$ (resp.\ $w_2$).
By substituting $A=\lambda I$, $B=I$ we get $2\lambda^{n_1}=2\lambda^{n_2}$. This is only possible if $n_1=n_2$. The same argument shows that $w_1$ and $w_2$ must have the same number of $b$'s as well. That is, the words are permutations of one another. Therefore, for all $A$, $B \in \mathcal{M}_2(\mathbb{C})$, the matrices $w_1(A,B)$ and $w_2(A,B)$ have not only the same trace, but also the same determinant. So, they have  the same eigenvalues with the same multiplicities.
\end{proof}

Given a word $w \in \mathbb{F}^+_2$, then the \emph{reverse} or \emph{mirror image} of $w$ is the word $\widetilde{w} \in \mathbb{F}^+_2$ obtained by writing the letters of $w$ in opposite order. So $w$ is a palindrome if and only if $\widetilde{w} = w$.

\begin{proposition}\label{p.mirror}
Every word $w \in \mathbb{F}^+_2$ is $2$-isospectral to its mirror image $\widetilde{w}$.
\end{proposition}

This \lcnamecref{p.mirror} is due to Southcott \cite[Theorem~6.3]{Southcott}. Let us provide a direct proof: 

\begin{proof}
For any two matrices $A$, $B$, the two matrix products $w(A,B)$ and $\widetilde{w}(A,B)$ have the same determinant, and to prove the \lcnamecref{p.mirror} we need to show that they have the same trace. 
That is, we need to show that the Fricke polynomials  $F_{\widetilde{w}}$ and $F_w$ are identical. 
Note that: 
\begin{equation}
\widetilde{w}(A,B) = [w(A^\mathtt{t},B^\mathtt{t})]^\mathtt{t} \, ,
\end{equation}
where the superscript denotes matrix transposition.
Using the fact that transposition preserves trace and determinant, we have: 
\begin{align}
\tr \widetilde{w}(A,B) 
&= \tr w(A^\mathtt{t},B^\mathtt{t}) \\
&= F_w ( \tr A^\mathtt{t}, \tr B^\mathtt{t}, \tr A^\mathtt{t} B^\mathtt{t}, \det A^\mathtt{t}, \det B^\mathtt{t}) \\
&= F_w ( \tr A, \tr B, \underbrace{\tr B A}_{\tr AB}, \det A, \det B) \\ 
&= \tr w(A,B)  \, ,
\end{align}
proving that $F_{\widetilde{w}} \equiv F_w$.
\end{proof}

\cref{p.mirror} directly implies not only \cref{l.SC} but many other similar relations. The proposition is only interesting if the mirror is \emph{not} a cyclic permutation of the original word. Thus we introduce the following definition.
\begin{definition}
We call a word $w \in \mathbb{F}^+_2$ \emph{chiral} if $\tilde{w}$ is not a cyclic permutation of~$w$.
\end{definition}
The shortest chiral words have length $6$, and are exactly $w  = a^2bab^2$, $\tilde w = b^2aba^2$, and their cyclic permutations.
Modulo cyclic permutations and interchanging the two letters, there exists only one chiral pair of length $7$, namely $a^3 b a b^2$ and $b^2aba^3$.
As the length increases, chiral words become increasingly frequent. For example, $61 \%$ of the words of length $10$ and $97 \%$ of the words of length $20$ are chiral (see \cite{OEIS}).

\subsection{Further notes}\label{ss.further}

\cref{p.mirror} is only the tip of the $2$-isospectrality iceberg.
Another way of producing $2$-isospectral words is as follows: if $w_1$ and $w_2$ are $2$-isospectral, and we substitute the letters $a$ and $b$ by any pair of words, then we obtain another pair of $2$-isospectral words. For instance,
performing the substitution $a\to ab$ and $b\to ba$ on the shortest chiral  pair
$w_1 = a^2bab^2$ and $w_2 = \widetilde{w_1} = b^2aba^2$,
we obtain the words 
\begin{equation}
abab^2a^2b^2aba \quad\text{and}\quad
baba^2b^2a^2bab \, ,
\end{equation}
which are $2$-isospectral despite not being cyclic permutations nor mirror images of one another. As remarked in \cite[\S~4.3]{LLM},
there is no shorter pair with those properties.

It seems to be a difficult problem to describe in purely combinatorial terms the relation of $2$-isopectrality on the semigroup $\mathbb{F}^+_2$. 
As for the free group $\mathbb{F}_2$, Horowitz \cite[Example 8.2]{Horowitz} proved that the $2$-isospectrality equivalence classes can be arbitrarily large. See  \cite[Conjecture~4.1]{Anderson} for a possible description of such classes.

Let us note that \cref{p.mirror} still holds
for the group $\mathbb{F}_2$.
However, it fails if three or more letters are used: for instance, $\tr(ABC) \not\equiv \tr(CBA)$.
On the other hand, we may produce non-trivial trace identities for three or more matrices using the substitution trick explained above.

As for higher dimensions, \cref{p.mirror} also fails if $d \ge 3$. No pair of $d$-isospectral words with $d\ge 3$ is known, except for cyclic permutations. In fact,  it is an open problem whether there exists a pair of words $w_1$, $w_2 \in \mathbb{F}_2^+$ that are not cyclic permutations of one another and satisfy $\tr w_1(A,B) = \tr w_2(A,B)$ for all $A$, $B \in \mathrm{SL}_3(\mathbb{C})$: see \cite[\S~4.3]{LLM}. If such words exist,  they have length bigger than $30$, according to experiments \cite{Reid}.

\section{Searching for matrix pairs with specific SMPs}\label{s.search}

\subsection{The procedure}

The example exhibited in \cref{s.example} was found by computer search.
Let us describe the procedure, which can be adapted to similar problems.

We start with a list of the Fricke polynomials of all Lyndon words on two letters up to a certain length $\ell$; we took $\ell = 14$, so $2538$ polynomials which can be computed using the algorithm described in \cite{Goldman}, for instance. After removing $2$-isospectral duplicates, we are left with $1549$ polynomials.

Instead of generating pairs of random matrices, we generate random values of the five variables $(x,y,z,u,v)$ as in \eqref{e.5_var}, with a conveniently chosen probability distribution. Specifically, we took each of the variables to be uniformly distributed on the interval $[-10,10]$, except for $z$, which we took uniformly distributed on $[-100,100]$.
Some $5$-tuples are not realizable by a pair of real matrices (a precise description being provided by \cref{p.real} below), and these can be discarded.
For each random $5$-tuple we can quickly compute the spectral radii of products of length up to $\ell$  using the Fricke polynomials.
Since we are looking for an example for which the shortest chiral words \eqref{e.SC_pair} are SMPs, we initially compute the normalized spectral radius of $A^2 B^2 AB$, and then compare it to the normalized spectral radius of the products in our list, immediately stopping if a better product is found (which happens in the vast majority of cases). If a viable candidate is found,  we produce actual matrices with the prescribed traces and determinants \eqref{e.5_var}, and then apply the polytope algorithm with balancing \cite{algo2} to confirm whether the pair of matrices does in fact make the products \eqref{e.SC_pair} spectrum maximizing.

In a specific run, we expanded the search to all chiral pairs of length up to $9$, in total $23$ pairs. We sampled $10^8$ five-tuples, from which $51$ viable candidates for non-uniquess were found. From those candidates, the polytope algorithm confirmed $20$ concrete instances of non-unique SMPs.

\subsection{More examples on non-unique SMPs}

We include more examples of 5-tuples $(x,y,z,u,v)$ corresponding to matrices $A$, $B$ with chiral (therefore non-unique) SMPs. The last column indicates that an invariant polytope with $2n$ vertices was found.

\begin{center}
\begin{tabular}{rrrrrrrr}
\thead{SMP} &\thead{$x$} &\thead{$y$}  &\thead{$z$}  &\thead{$u$}    &\thead{$v$}    &\thead{$n$} \\
$A^2B^2AB$ & $3.38477$  &$-0.84501$ &$5.58856$ &$4.29803$ &$5.99245$  &$18$\\
$A^2B^2AB$ & $-1.81325$  &$3.83802$ &$8.57711$ &$8.79352$ &$7.69271$  &$18$ \\
$A^3BAB^2$ & 
$-0.28009$  &$2.51662$ &$-9.78050$ &$7.09393$ &$3.76472$  &$34$\\
$A^3BAB^2$ & 
$-2.41561$  &$4.01089$ &$-10.27036$ &$8.39182$ &$4.16903$  &$40$\\
$A^2B^2AB^3$ & $-2.27713$ & $-4.85077$ & $-3.83135$ & $7.50043$ & $7.58161$ & $19$\\
$A^2B^2AB^3$ & $-2.46102$ & $-5.50086$ & $-4.86349$ & $9.90656$ & $9.80116$ & $23$\\
$A^2BA^2BAB^2$ & $2.48264$& $-0.68806$& $ 3.67748$& $2.74344$& $3.59137$& $30$\\
$A^2BA^2BAB^2$ & $3.16180$& $-0.93207$& $ 5.83803$& $4.74510$& $5.58561$& $30$
\end{tabular}
\end{center}

\subsection{Realizability of traces and determinants}

For every $5$-tuple $(x,y,z,u,v)$ of \emph{complex} numbers, it is easy to find complex $2 \times 2$ matrices satisfying the equations \eqref{e.5_var} (see the proof of \cref{p.Fricke}).
The situation of real matrices is more complicated, however:

\begin{proposition}\label{p.real}
Given $(x,y,z,u,v) \in \mathbb{R}^5$,
there exists a pair of matrices $(A,B) \in \mathcal{M}_2(\mathbb{R})^2$ solving the system of equations 
\begin{equation}\label{e.5_var_again}
\tr A = x, \quad
\tr B = y, \quad
\tr AB = z, \quad
\det A = u,\quad 
\det B = v 
\end{equation}
if and only if
the symmetric matrix
\begin{equation}\label{e.mystic}
M \coloneqq 
\begin{bmatrix}
u   & x/2   & z/2 \\
x/2 & 1     & y/2 \\
z/2 & y/2   & v
\end{bmatrix}
\end{equation}
is \emph{not} positive definite.
\end{proposition}

Recalling Sylvester's criterion for positive definiteness in terms of the positivity of any nested sequence of principal minors  \cite[Theorem~7.2.5]{HJ}, we see that the realizable $5$-tuples for the equations \eqref{e.5_var_again} can be described for instance as:
\begin{equation}
\min \big( 4u-x^2, 4uv-z^2  -vx^2-uy^2+xyz\big) \le 0 \, .
\end{equation}

\begin{proof}[Proof of \cref{p.real}]
Fix matrices $A$, $B$ and compute the corresponding five numbers \eqref{e.5_var_again} and the matrix $M$ from \eqref{e.mystic}.
Consider the following matrix pencil on three real variables:
\begin{equation}
P (\xi_1,\xi_2,\xi_3) \coloneqq \xi_1 A + \xi_2 I + \xi_3 \, \mathrm{adj}(B) =
\begin{bmatrix}
a_{11}\xi_1+\xi_2               +b_{22}\xi_3 & 
a_{12}\xi_1\phantom{\; +\xi_2}  -b_{12}\xi_3 \\
a_{21}\xi_1\phantom{\; +\xi_2} -b_{21}\xi_3 & 
a_{22}\xi_1+\xi_2               +b_{11}\xi_3 
\end{bmatrix}
\end{equation}
where $\mathrm{adj}(B)$ is the adjugate of $B$.
By inspection (or using \cite[eq.~(7), p.~9]{Fenchel}),
the matrix of the quadratic form 
\begin{equation}
Q(\xi_1,\xi_2,\xi_3) \coloneqq \det P(\xi_1,\xi_2,\xi_3) 
\end{equation}
is exactly $M$.
By dimension counting, we can find real numbers $\xi_1$, $\xi_2$, $\xi_3$, not all of them zero, such that the first column of the matrix $P(\xi_1,\xi_2,\xi_3)$ vanishes, and in particular $Q(\xi_1,\xi_2,\xi_3) = 0$.
So the quadratic form $Q$ cannot be positive definite, proving the ``only if'' part of the \lcnamecref{p.real}.

For the converse, consider a positive definite matrix $M$ of the form \eqref{e.mystic}, and let us exhibit a solution $(A,B)$ for the system of equations \eqref{e.5_var_again}.

First, consider the case where $4u-x^2 \le 0$. Then we can find real numbers $\lambda_1$, $\lambda_2$ whose sum is $x$ and whose product is $u$.
Take $A = \left[ \begin{smallmatrix} \lambda_1 & \mu \\ 0 & \lambda_2 \end{smallmatrix} \right]$, where $\mu$ is a free parameter. On the other hand, choose a matrix $B = \left[ \begin{smallmatrix} a & b \\ c & d \end{smallmatrix} \right]$ with $\tr B = y$ and $\det B = v$, taking the precaution that $c \neq 0$. Then
\begin{equation}
\tr AB = \lambda_1 a + \mu c + \lambda_2 d \, .
\end{equation}
So there exists $\mu$ meeting the condition $\tr AB = z$, and we found a solution for the system \eqref{e.5_var_again}.

The case where $4v-y^2 \le 0$ is entirely analogous, reversing the roles of the two matrices.

Finally, assume that $4u-x^2>0$ and $4v-y^2>0$.
In this case, $\det M \le 0$, otherwise the matrix $M$ would be positive definite by Sylvester's criterion.
Let
\begin{equation}\label{eliptic}
A = 
\begin{bmatrix}
x/2 & -\sqrt{u-x^2/4} \\ 
\sqrt{u-x^2/4} & x/2 
\end{bmatrix}, \ 
B =
\begin{bmatrix}
y/2 & -\mu^{-1}\sqrt{v-y^2/4} \\ 
\mu\sqrt{v-y^2/4} & y/2 
\end{bmatrix},
\end{equation}
where $\mu$ is a free parameter.
These matrices have the desired traces and determinants, so we only need to worry about the trace of their product:
\begin{equation}
\tr AB = \frac{xy}{2} - \left(\mu + \mu^{-1}\right) \sqrt{u-\frac{x^2}{4}}\sqrt{v-\frac{y^2}{4}} \, .
\end{equation}
Equating this to $z$, we obtain a quadratic equation $a \mu^2 + b\mu + c = 0$, where 
\begin{equation}
a=c=\sqrt{u-\frac{x^2}{4}}\sqrt{v-\frac{y^2}{4}} \, , \quad b = z - \frac{xy}{2} \, .
\end{equation}
As a computation shows, the discriminant $b^2-4ac$ is exactly $-4 \det M$, hence nonnegative.
So a real solution $\mu$ exists.
Note that $\mu \neq 0$, so our matrices $A$ and $B$ are well-defined, and solve equations \eqref{e.5_var_again}.
\end{proof}

\section{Questions}\label{s.questions}

Our findings suggest a few questions. 
The main finding of this paper is the negative answer to  \cref{q.unique} in dimension $d=2$ (not only for pairs of matrices, but for families of any finite cardinality $k \ge 2$: see \cref{ss.additional}). What about higher dimensions? As discussed in \cref{ss.further}, there are no candidate words that could violate uniqueness. Therefore, \cref{q.unique} is entirely open when $d\ge 3$.

For the rest of this section, let us confine ourselves to the simpler setting of pairs of real $2 \times 2$ matrices.

Recall that $2$-isospectral words $w_1$, $w_2\in \mathbb{F}_2^+$ are permutations of one another, and in particular have the same length. Therefore, given a pair $(A,B)$  of $2 \times 2$ matrices, if the product $\Pi = w_1(A,B)$ is an SMP, then so is the  product $\Pi' = w_2(A,B)$. To put in another way, $2$-isospectrality is a natural enemy to uniqueness of SMPs. To take that fact into account, \cref{q.unique} (specialized to pairs of $2\times 2$ matrices) needs to be ``corrected'' as follows:

\begin{question}\label{q.rel_unique}
Does Lebesgue-almost every pair family of $2 \times 2$ matrices have a unique SMP, up to $2$-isospectrality?
\end{question}

Since we do not know if existence of SMP is a typical property (Maesumi conjecture), let us formulate a simpler version of the question above:

\begin{question}
Is it true that for Lebesgue almost every pair $(A,B)$ of real $2 \times 2$ matrices, all SMPs (if they exist) are $2$-isospectral to one another?
\end{question}

The question above could be accessible if the property of $2$-isospectrality were better understood.

Next, we pose a different kind of problem.

\begin{question}\label{q.winners}
Which words in two letters can be SMPs of a pair of $2 \times 2$ matrices with the relative uniqueness property?
\end{question}

For example, every finite sturmian word appears as a unique SMP of some pair of matrices: see \cite{MS}.
The same is true for words of the form $a^m b^n$, though we do not want to extend this note by including a proof. 
On the opposite direction, it seems reasonable to expect that many (if not most) words do \emph{not} have the property alluded in \cref{q.winners}, but nevertheless we are currently unable to exhibit a single one. 

As mentioned in the introduction, even if a pair of matrices does not admit a SMP (that is, it is a counterexample to the finiteness conjecture), it always admits a Lyapunov-maximizing measure $\mu$, which is a ergodic shift-invariant probability measure on the space of infinite strings of two letters for which the first Lyapunov exponent is maximized.  Note that a Lyapunov-maximizing measure is atomic if and only if it comes from an SMP.

\begin{question}
Suppose that $\mu$ is a \emph{non-atomic} ergodic measure for the shift, and let $M_\mu$ be the set of pair of matrices for which $\mu$ is Lyapunov-maximizing. Does $M_\mu$ have empty interior?
\end{question}

An analogous result for Birkhoff averages is due to Yuan and Hunt \cite{YH}, and in fact was one of the first theoretical results in commutative ergodic optimization.

\section*{Acknowledgements}
We thank Alexander Christie, Sean Lawton, and Peter Selinger for sharing information about trace identities. 
We are grateful to Thomas Mejstrik for helpful discussions and assistance with his invariant polytope software \cite{thomas}.
We thank Ian D.\ Morris for several valuable comments and suggestions. 
We also thank the referees for corrections and improvements. Our original polytope had $18 \times 2$  vertices, and a referee noted that the same pair of matrices admits an invariant polytope with $16 \times 2$ vertices.


\begin{thebibliography}{99} %: bibliography


\bibitem{Anderson}
\textsc{Anderson, James W.} -- 
Variations on a theme of Horowitz. In: Komori, Y.\ (ed.) et al., \textit{Kleinian groups and hyperbolic 3-manifolds.} Proceedings of the Warwick workshop, Warwick, UK, September 11–14, 2001. Cambridge: Cambridge University Press. 
Lond.\ Math.\ Soc.\ Lect.\ Note Ser.\ 299, 307--341 (2003).
\Zbl{1046.57015}
\MRev{2044556}

\bibitem{berger-wang}
\textsc{Berger, Marc A.; Wang, Yang} --
Bounded semigroups of matrices.
\textit{Linear Algebra Appl.\ } 166, 2127 (1992).
\doi{10.1016/0024-3795(92)90267-E}
\Zbl{0818.15006}
\MRev{1152485}


\bibitem{BTV}
\textsc{Blondel, Vincent D.; Theys, Jacques; Vladimirov, Alexander A.} --
An elementary counterexample to the finiteness conjecture.
\textit{SIAM J.\ Matrix Anal.\ Appl.\ }24, no.\ 4, 963--970 (2003).
\doi{10.1137/S0895479801397846}
\Zbl{1043.15007}
\MRev{2003315}

\bibitem{undecidable}
\textsc{Blondel, Vincent; Tsitsiklis, John} --
 The boundedness of all products of a pair of matrices is undecidable.
\textit{Syst.\ Control Lett.\ } 41, no.\ 2, 135--140 (2000).
\doi{10.1016/S0167-6911(00)00049-9}
\Zbl{0985.93042}
\MRev{1831027}

\bibitem{BochiICM}
\textsc{Bochi, Jairo} -- 
Ergodic optimization of Birkhoff averages and Lyapunov exponents.
Sirakov, Boyan (ed.) et al., 
\textit{Proceedings of the ICM 2018},  Volume III. Invited lectures. 
Hackensack, NJ: World Scientific. 1825--1846 (2018).
\Zbl{1448.37059}
\MRev{3966831}

\bibitem{finiteness}
\textsc{Bousch, Thierry; Mairesse, Jean} --
Asymptotic height optimization for topical IFS, Tetris heaps, and the finiteness conjecture.
\textit{J.\ Amer.\ Math.\ Soc.\ }15, 77--111 (2002).
\doi{10.1090/S0894-0347-01-00378-2}
\Zbl{1057.49007}
\MRev{1862798}

\bibitem{Breuillard}
\textsc{Breuillard, Emmanuel} --  On the joint spectral radius. In: Avila, Rassias, and Sinai (eds), \emph{Analysis at Large.} Springer, Cham. 1--16 (2022). \doi{10.1007/978-3-031-05331-3_1}
\Zbl{7556594}

\bibitem{Chas}
\textsc{Chas, Moira} -- 
\href{www.math.stonybrook.edu/~moira/NewApps/Fricke Polynomial/fricke.html}{www.math.stonybrook.edu/$\sim$moira/NewApps/Fricke Polynomial/fricke.html}

\bibitem{CGP}
\textsc{Cicone, Antonio; Guglielmi, Nicola; Protasov, Vladimir Yu.} -- 
Linear switched dynamical systems on graphs.
\textit{Nonlinear Anal.\ Hybrid Syst.\ }29, 165--186 (2018).
\doi{10.1016/j.nahs.2018.01.006}
\Zbl{1391.37035}
\MRev{3795595}

\bibitem{Contreras}
\textsc{Contreras, Gonzalo} -- 
Ground states are generically a periodic orbit.
\textit{Invent.\ Math.\ }205, no.\ 2, 383--412 (2016).
\doi{10.1007/s00222-015-0638-0}
\Zbl{1378.37047}
\MRev{3529118}

\bibitem{DaubechiesL}
\textsc{Daubechies, Ingrid; Lagarias, Jeffrey C.} -- 
Sets of matrices all infinite products of which converge.
\textit{Linear Algebra Appl.\ }161, 227--263 (1992).
\doi{10.1016/0024-3795(92)90012-Y}
\Zbl{0746.15015}
\MRev{1142737}

\bibitem{Fenchel}
\textsc{Fenchel, Werner} -- 
\textit{Elementary geometry in hyperbolic space.} 
De Gruyter Studies in Mathematics, 11. Berlin etc.: Walter de Gruyter \& Co.\ (1989).
\Zbl{0674.51001}
\MRev{1004006}

\bibitem{Formanek}
\textsc{Formanek, Edward} -- 
Polynomial identities and the Cayley-Hamilton theorem.
\textit{Math.\ Intell.\ }11, no.\ 1, 37--39 (1989).
\doi{10.1007/BF03023774}
\Zbl{0667.16013}
\MRev{0979022}

\bibitem{Goldman}
\textsc{Goldman, William M.} -- 
An exposition of results of Fricke and Vogt. 
\arxiv{math/0402103}

\bibitem{algo1}
\textsc{Guglielmi, Nicola; Protasov, Vladimir} --
Exact computation of joint spectral characteristics of linear operators.
\textit{Found.\ Comput.\ Math.\ }13, no.\ 1, 37--97 (2013).
\doi{10.1007/s10208-012-9121-0}
\Zbl{1273.65054}
\MRev{3009529}

\bibitem{algo2}
\textsc{Guglielmi, Nicola; Protasov, Vladimir} --
Invariant polytopes of sets of matrices with applications to regularity of wavelets and subdivisions.
\textit{SIAM J.\ Matr.\ Anal.\ Appl.\ } 37, no.\ 1, 18--52  (2016).
\doi{10.1137/15M1006945}
\Zbl{1382.15033}
\MRev{3447127}

\bibitem{GWZ2005}
\textsc{Guglielmi, Nicola; Wirth, Fabian; Zennaro, Marino} --
Complex polytope extremality results for families of matrices.
\textit{SIAM J.\ Matr.\ Anal.\ Appl.\ } 7, no.\ 3, 721–-743  (2005).
\doi{10.1137/040606818}
\Zbl{1099.15023}
\MRev{2208331}

\bibitem{absco}
\textsc{Guglielmi, Nicola; Zennaro, Marino} -- 
Balanced complex polytopes and related vector and matrix norms.
\textit{J.\ Convex Anal.\ }14, no.\ 4, 729--766 (2007).
\Zbl{1128.52010}
\MRev{2350813}

\bibitem{GZ15}
\textsc{Guglielmi, Nicola; Zennaro, Marino} -- 
Canonical construction of polytope Barabanov norms and antinorms for sets of matrices.
\textit{SIAM J.\ Matr.\ Anal.\ Appl.\ }36, no.\ 2, 634–-655  (2015).
\doi{10.1137/140962814}
\Zbl{1320.15017}
\MRev{3352607}


\bibitem{Hare}
\textsc{Hare, Kevin G.; Morris, Ian D.; Sidorov, Nikita; Theys, Jacques} --
An explicit counterexample to the Lagarias-Wang finiteness conjecture.
\textit{Adv.\ Math.\ }226, 4667--4701 (2011).
\doi{10.1016/j.aim.2010.12.012}
\Zbl{1218.15005}
\MRev{2775881}

\bibitem{HJ}
\textsc{Horn, Roger A.; Johnson, Charles R.} --
\textit{Matrix analysis.} 2nd ed.
Cambridge: Cambridge University Press (2013).
\Zbl{1267.15001}
\MRev{2978290}

\bibitem{Horowitz}
\textsc{Horowitz, Robert D.} -- 
Characters of free groups represented in the two-dimensional special linear group.
\textit{Commun.\ Pure Appl.\ Math.\ }25, 635--649 (1972).
\doi{10.1002/cpa.3160250602}
\Zbl{1184.20009}
\MRev{0314993}

\bibitem{HuntOtt}
\textsc{Hunt, Brian R.; Ott, Edward} -- 
Optimal periodic orbits of chaotic systems occur at low period.
\textit{Phys.\ Rev.\ E }54, no.\ 1, 328--337 (1996).
\doi{10.1103/PhysRevE.54.328}

\bibitem{Jenk}
\textsc{Jenkinson, Oliver} -- 
Ergodic optimization.
\textit{Discrete Contin.\ Dyn.\ Syst.\ }15, no.\ 1, 197--224 (2006).
\doi{10.3934/dcds.2006.15.197}
\Zbl{1116.37017}
\MRev{2191393}

\bibitem{JenkinsonPollicott}
\textsc{Jenkinson, Oliver;  Pollicott, Mark} --
Joint spectral radius, Sturmian measures and the finiteness conjecture. 
\textit{Ergodic Theory Dyn.\ Syst.\ }38, no.\ 8, 3062--3100 (2018).
\doi{10.1017/etds.2017.18}
\Zbl{1405.15028}
\MRev{3868023}

\bibitem{Jorgensen}
\textsc{Jørgensen, Troels} -- 
Traces in $2$-generator subgroups of $\mathrm{SL}(2,\mathbf{C})$. 
\textit{Proc.\ Am.\ Math.\ Soc.\ }84, 339--343 (1982).
\doi{10.2307/2043555}
\Zbl{0498.20035}
\MRev{0640226}

\bibitem{Jungers}
\textsc{Jungers, Raphaël} -- 
\textit{The joint spectral radius: Theory and applications.}
Lecture Notes in Control and Information Sciences, 385. 
Berlin: Springer-Verlag (2009).
\MRev{MR2507938}

\bibitem{nonalgebraic}
\textsc{Kozyakin, V.S.} -- 
Algebraic unsolvability of problem of absolute stability of desynchronized systems.
\textit{Autom.\ Remote Control} 51, no.\ 6, 754--759 (1990).
\Zbl{0737.93056}
\MRev{1071607}

\bibitem{Kozyakin}
\textsc{Kozyakin, V.} --
A dynamical systems construction of a counterexample to the finiteness conjecture. \textit{Proceedings of the 44th IEEE Conference on Decision and Control} (2005), pp. 2338--2343.
\doi{10.1109/CDC.2005.1582511}

\bibitem{Wang}
\textsc{Lagarias, Jeffrey C.; Wang, Yang} --
The finiteness conjecture for the generalized spectral radius of a set of matrices.
\textit{Linear Algebra Appl.\ }214 (1995), 17--42.
\doi{10.1016/0024-3795(93)00052-2}
\Zbl{0818.15007}
\MRev{1311628}

\bibitem{LLM}
\textsc{Lawton, Sean; Louder, Larsen; McReynolds, D.B.} -- 
Decision problems, complexity, traces, and representations. \textit{Groups Geom.\ Dyn.\ }11, no.\ 1, 165--188 (2017).
\doi{10.4171/GGD/393}
\Zbl{1423.20023}
\MRev{3641838}

\bibitem{roots}
\textsc{Lischke, Gerhard} -- 
Primitive words and roots of words.
\textit{Acta Univ.\ Sapientiae, Inform.\ }3, no.\ 1, 5--34 (2011).
\Zbl{1234.68224}

\bibitem{Maesumi}
\textsc{Maesumi, Mohsen} --
 Optimal norms and the computation of joint spectral radius of matrices. \textit{Linear Algebra Appl.\ } 428, 2324--2338 (2008). \doi{10.1016/j.laa.2007.09.036}
 \Zbl{1138.65030}
 \MRev{2408030}

\bibitem{Magnus}
\textsc{Magnus, W.} -- 
The uses of $2$ by $2$ matrices in combinatorial group theory. A survey.
\textit{Result.\ Math.\ }4, 171--192 (1981).
\doi{10.1007/BF03322976}
\Zbl{0468.20031}
\MRev{636484}



\bibitem{Mejstrik}
\textsc{Mejstrik, Thomas} -- 
Algorithm 1011: Improved invariant polytope algorithm and applications. 
\textit{ACM Trans.\ Math.\ Softw.\ }46, no.\ 3, Article no.\ 29, 26 p. (2020).
\doi{10.1145/3408891}
\Zbl{1484.65090}
\MRev{4161245}

\bibitem{thomas}
\textsc{Mejstrik, Thomas} --
ttoolboxes software, \href{http://gitlab.com/tommsch/ttoolboxes}{http://gitlab.com/tommsch/ttoolboxes}.

\bibitem{Morris_Mather}
\textsc{Morris, Ian D.} --
 Mather sets for sequences of matrices and applications to the study of joint spectral radii.
\textit{Proc.\ Lond.\ Math.\ Soc.\ }(3) 107, no.\  1, 121--150 (2013).
\doi{10.1112/plms/pds080}
\Zbl{1277.15009}
\MRev{3083190}


\bibitem{Morris_prevalent}
\textsc{Morris, Ian D.} -- 
Prevalent uniqueness in ergodic optimisation.
\textit{Proc.\ Am.\ Math.\ Soc.\ }149, no.\ 4, 1631--1639 (2021).
\doi{10.1090/proc/15334}
\Zbl{1466.37005}
\MRev{4242318}

\bibitem{MS}
\textsc{Morris, Ian D.; Sidorov, Nikita} -- 
On a devil’s staircase associated to the joint spectral radii of a family of pairs of matrices. 
\textit{J.\ Eur.\ Math.\ Soc.\ }15, no.\ 5, 1747--1782 (2013).
\doi{10.4171/JEMS/402}
\Zbl{1329.15030}
\MRev{3082242}

\bibitem{OEIS}
\textit{The On-Line Encyclopedia of Integer Sequences}, sequences \OEIS{A1371} and \OEIS{A1037}.

\bibitem{ME}
\textsc{pgassiat} -- $\operatorname{tr}(AABABB) = \operatorname{tr}(AABBAB)$ for $2 \times 2$ matrices.
\textit{Mathematics StackExchange}, 
\href{http://math.stackexchange.com/q/29078}{http://math.stackexchange.com/q/29078}
(version: 2011-12-04).

\bibitem{PW}
\textsc{Plischke, Elmar; Wirth, Fabian} --
Duality results for the joint spectral radius and transient behavior.
\textit{Linear Algebra Appl.\ }428, no.\ 10, 2368--2384 (2008).
\doi{10.1016/j.laa.2007.12.009}
\Zbl{1330.15010}
\MRev{2408033}



\bibitem{Protasov96}
\textsc{Protasov, V.\ Yu.\ } -- 
The joint spectral radius and invariant sets of linear operators (in Russian).
\textit{Fundam.\ Prikl.\ Mat.\ }2, no.\ 1, 205--231 (1996).
\Zbl{0899.47002}
\MRev{1789006}

\bibitem{Protasov21}
\textsc{Protasov, Vladimir Y.} --
The Barabanov norm is generically unique, simple, and easily computed.
\textit{SIAM J.\ Control Optim.\ }60, no.\ 4, 2246--2267 (2022).
\doi{10.1137/21M1426821}
\Zbl{1497.93137}
\MRev{4457709}

\bibitem{Reid}
\textsc{Reid, Tim} -- 
Special words in free groups.
\textit{Mason Experimental Geometry Lab.} 
\href{http://meglab.wikidot.com/local--files/research:spring2017/Tim_Streets_Talk.pdf}{http://meglab.wikidot.com/local--files/research:spring2017/Tim\_Streets\_Talk.pdf}
(2017). 


\bibitem{rota-strang}
\textsc{Rota, Gian-Carlo; Strang, W. Gilbert} --
A note on the joint spectral radius.
\textit{Kon.\ Nederl.\ Acad.\ Wet.\ Proc.\ }63 (1960), 379--381.
\doi{10.1016/S1385-7258(60)50046-1}
\Zbl{0095.09701}
\MRev{0147922}

\bibitem{Selinger}
\textsc{Selinger, Peter} --
Finite dimensional Hilbert spaces are complete for dagger compact closed categories.
\textit{Log.\ Methods Comput.\ Sci.\ } 8 (2012), paper no.\ 6, 12 pp.
\doi{10.2168/LMCS-8(3:6)2012}
\Zbl{1250.18007}
\MRev{2965391}

\bibitem{Southcott}
\textsc{Southcott, J.B.} --
Trace polynomials of words in special linear groups. \textit{J.\ Aust.\ Math.\ Soc., Ser.\ A} 28, 401--412 (1979). 
\doi{10.1017/S1446788700012544} 
\Zbl{0424.20040}
\MRev{0562872}

\bibitem{Sylvester}
\textsc{Sylvester, J.J.} -- On the involution of two matrices of the second order. \textit{British Association Report, Southport} (1883). In: \textit{The Collected Mathematical Papers of James Joseph Sylvester}, volume IV, Baker, H.F. (ed.), Cambridge University Press (1912), pp.\ 115--117.
\Zbl{43.0026.01}

\bibitem{hard}
\textsc{Tsitsiklis, John N.; Blondel, Vincent D.} -- 
The Lyapunov exponent and joint spectral radius of pairs of matrices are hard - when not impossible - to compute and to approximate.
\textit{Math.\ Control Signals Syst.\ }10, no.\ 1, 31--40 (1997).
\doi{10.1007/BF01219774}
\Zbl{0888.65044}
\MRev{1486730}

\bibitem{VGJ}
\textsc{Vankeerberghen, Guillaume; Hendrickx, Julien; Jungers, Raphaël M.} -- 
JSR: a toolbox to compute the joint spectral radius.
\textit{HSCC'14 -- Proceedings of the 17th ACM international conference on hybrid systems: computation and control} %, HSCC 2014, Berlin, Germany, April 15–17, 2014. 
New York, NY: Association for Computing Machinery. 151--156 (2014).
\doi{10.1145/2562059.2562124}
\Zbl{1364.65099}



\bibitem{YH}
\textsc{Yuan, Guocheng; Hunt, Brian R.} -- 
Optimal orbits of hyperbolic systems.
\textit{Nonlinearity} 12, no.\ 4, 1207--1224 (1999).
\doi{10.1088/0951-7715/12/4/325}
\Zbl{0951.37006}
\MRev{1709845}




\end{thebibliography}
\end{document}